\documentclass[12pt]{article}

\usepackage{amsfonts,amsmath, amssymb}

\usepackage{xcolor}

\usepackage{enumerate}
\usepackage{epsf,epsfig,amsfonts,a4wide}
\usepackage{amsfonts, mathdots}
\usepackage{amsmath,amssymb,amsthm}
\usepackage{url}
\usepackage{amsmath,amssymb,amsthm}

\usepackage{amssymb}
\usepackage{amsthm}
\usepackage{epsf,epsfig,amsfonts,graphicx}

\newtheorem{Theorem}{Theorem}[section]

\newtheorem{Lemma}{Lemma}[section]
\newtheorem{Proposition}{Proposition}[section]
\newtheorem{Corollary}{Corollary}[section]

\newtheorem{Remark}{Remark}[section]

\theoremstyle{remark}

\newcommand{\be}{\begin{equation}}
\newcommand{\ee}{\end{equation}}

\newcommand{\R}{\mathbb{R}}\newcommand{\Id}{\textrm{\rm Id}}

\newcommand{\gl}{\mathrm{gl}}

\newcommand{\ddd}{\mathrm{d}}

\newcommand{\pd}[2]{\frac{\partial#1}{\partial#2}}

\newcommand{\tr}{\operatorname{tr}}

\newcommand{\weg}[1]{}

\title{Applications of Nijenhuis Geometry V: geodesically equivalent  metrics and finite-dimensional reductions of certain integrable quasilinear systems}
\author{Alexey V. Bolsinov\footnote{ School of Mathematics,
 Loughborough University,
 LE11 3TU, UK \ \ 
 \quad {\tt A.Bolsinov@lboro.ac.uk} } \quad
\& \quad  Andrey Yu. Konyaev\footnote{Faculty of Mechanics and Mathematics, Moscow State University,  and  Moscow Center for Fundamental and Applied Mathematics,119992, Moscow Russia
 \ \ \quad {\tt  maodzund@yandex.ru}} \quad \& \quad Vladimir S. Matveev\footnote{
Institut f\"ur Mathematik, Friedrich Schiller Universit\"at Jena,
07737 Jena Germany  and  La Trobe University, Melbourne, Australia \ \ \quad {\tt  vladimir.matveev@uni-jena.de}} 
}  

\begin{document}

\maketitle
\begin{abstract}
We describe all metrics geodesically compatible with a $\gl$-regular Nijenhuis operator $L$. The set of such metrics is large enough so that a generic local curve $\gamma$ is a geodesic for a suitable metric $g$ from this set. Next, we show that   a certain evolutionary PDE system of hydrodynamic type constructed from $L$ preserves the property of $\gamma$ to be a $g$-geodesic.
This implies that  every metric $g$ geodesically compatible with $L$ gives us a finite dimensional reduction of this PDE system. We show that its restriction onto the set of $g$-geodesics is naturally equivalent to the Poisson action of $\mathbb{R}^n$ on the cotangent bundle generated   by the integrals coming from geodesic compatibility.  
\end{abstract}


\section{Introduction }

This work continues the research program started in  \cite{openprob, nij}.
The main object of study within this program  are $(1,1)$-tensor fields with vanishing Nijenhuis torison known as {\it Nijenhuis operators}.  They pop up in many, a priori  unrelated, branches  of mathematics, so it makes sense  to develop  a general theory of Nijenhuis operators and then to apply the results and methods obtained wherever these operators appear (e.g., in the theory of geodesically equivalent metrics  as in the present paper). This approach treats  a Nijenhuis operator as a {\it primary} object, even if it initially appeared  as a {\it secondary} object in the study of another structure.  An unexpected positive outcome of this change of perspective is that it reveals hidden relationships between different subjects.
A demonstration of this phenomenon is the paper \cite{nijapp2} showing that pencils of  compatible  Poisson brackets of hydrodynamic type are closely related to geodesically equivalent metrics of constant curvature.  This allowed us to apply methods and results of a more developed theory of geodesically equivalent metrics to the theory of  geometric Poisson structures. Another example is an unexpected relation between compatible geometric Poisson structures  of type $\mathcal{P}_3 +\mathcal{P}_1$ and orthogonal  separation of variables for spaces of constant curvature, 
which allowed us to obtain new results in both subjects by combining the relevant ideas and methods. In the theory of separation of variables, this led us in  \cite{ortsepar}
to a description of all orthogonal separating coordinates for pseudo-Riemannian spaces of constant curvature,  solving a long-standing problem going back to L. Eisenhart \cite{eisenhart}.   On the other hand, in  \cite{nijapp2} we have constructed all non-degenerate compatible Poisson structures of the type $\mathcal{P}_3 +\mathcal{P}_1$  such that $\mathcal{P}_3$ is Darboux-Poisson, which, in turn, led us to a construction of new integrable PDE systems in  \cite{nijapp4}. 

In the present paper, the {\it partner structure} for  a  Nijenhuis operator $L$ is a metric $g$ which is geodesically compatible to it.   In Theorem 
\ref{thm:2}, we relate such  metrics to the symmetries of the operator $L$. Combining this relation with  results of \cite{nij4}, we describe all metrics geodesically compatible with a $\gl$-regular Nijenhuis operator and, in particular, show that  every such operator locally admits a geodesically compatible metric, see Theorem \ref{thm:1}.  Next, we consider an integrable PDE system of hydrodynamic type constructed from a $\gl$-regular  Nijenhuis operator $L$, which was studied (mostly, in the diagonal case)  in many papers including \cite{nijapp4, fer,  ml, MB2010}.  Theorem \ref{thm:3a} shows  that this system preserves the property of a curve to be a geodesic of any fixed geodesically compatible metric. In other words, every choice of  a geodesically compatible partner for $L$ on a manifold $\mathsf M$ gives us a finite-dimensional reduction of the system. Finally, Theorem \ref{thm:3b} states  that the   
corresponding finite-dimensional reduction is  naturally equivalent to the Poisson action on $T^*\mathsf{M}$ generated by the commuting integrals coming from geodesic compatibility. 

{\bf Acknowledgements.} Vladimir Matveev  thanks the DFG for the support (grant  MA 2565/7). Some of results were obtained during a long-term research visit of VM to La Trobe University supported by the  Sydney Mathematics Research Institute  and the ARC Discovery Programme   DP210100951.

\subsection{Definitions and  results}

Two (pseudo)-Riemannian metrics $g$ and $\bar g$ (of any, possibly different, signatures) are called {\it geodesically equivalent} if they share the same  geodesics viewed as unparameterized curves. According to \cite{eqm},  a manifold endowed with a pair of such metrics carries a natural Nijenhuis structure defined by the operator 
$$
L = \left|\frac{\det \bar g}{\det g}\right|^{\frac{1}{n+1}}   \bar g^{-1}  g.
$$
Since $\bar g$ is uniquely reconstructed from $L$  as $\bar g = \frac{1}{|\det L|} gL^{-1}$, the study of geodesically equivalent metrics reduces to the study of pairs $(g, L)$  satisfying the following compatibility condition:
a metric $g$ and a Nijenhuis operator $L$ are said to be {\it geodesically compatible}, if $L$ is $g$-self-adjoint and the metric $\bar g = \frac{1}{|\det L|} gL^{-1}$ is geodesically equivalent to $g$.

Analytically,  the geodesic compatibility condition is given by the PDE equation   \cite{eqm, sinjukov}
\begin{equation}
\label{eq:geodeqv0}
\nabla_\eta L = \frac{1}{2} \bigl(  \eta\otimes \ddd\tr L + (\eta\otimes \ddd\tr L)^*   \bigr),
\end{equation}
where $\eta$ is an arbitrary vector field.  Notice that this relation is linear in $L$. Our first result is an equivalent version of \eqref{eq:geodeqv0}, which is linear both in $L$ and in $g$  and contains no covariant derivative.

\begin{Theorem}\label{thm:0}
An operator $L$ and a metric $g$ are geodesically compatible if and only if $L$ is $g$-self-adjoint and the following relation holds 
\begin{equation}
\label{eq:geodeqv4bis}
 \mathcal L_{L\xi} \bigl(g(  \eta, \xi)\bigr) -  \mathcal L_\xi \bigl(g ( \eta, L\xi)\bigr) - g\bigl(  \eta, [L\xi ,\xi]\bigr)   + g \bigl( [\eta, L\xi], \xi \bigl) -  g \bigl( [\eta,\xi], L\xi\bigr)\,  =  g( \eta  , \xi)\,   \mathcal L_\xi \tr L,
 \end{equation}
for any vector fields $\xi$ and $\eta$. In local coordinates,  \eqref{eq:geodeqv4bis} is equivalent to 
\begin{equation}
\label{eq:c1bis}
 g_{k\alpha} \frac{\partial L^\alpha_j}{\partial x^i} 
+ \left( \frac{\partial g_{ik}}{\partial x^\alpha}   -  \frac{\partial g_{i\alpha}}{\partial x^k} \right) L^\alpha_j  {\color{blue}  \ +  \  (j \leftrightarrow k)} \, = \,    \, g_{ik}  \, \frac{\partial \tr L}{\partial x^j} {\color{blue} \ +  \ (j \leftrightarrow k)} 
\end{equation}
where $(j \leftrightarrow k)$ denotes the expression obtained by interchanging  indices $j$ and $k$. 
\end{Theorem}

Local description of geodesically compatible pairs $(g,L)$  is known at algebraically generic points, near which the algebraic type of $L$ does not change.    For diagonalisable operators $L$, it follows from the classical work by T. Levi-Civita \cite{LC}, the general case was done in \cite{tams}.   However, for singular points at which the algebraic type of $L$ changes (e.g., the eigenvalues of $L$ collide), the description of Nijenhuis operators $L$ admitting  at least one geodesically equivalent partner $g$ remains an open problem.  Certain restrictions definitely exist. For instance, the operator $L = \begin{pmatrix} x & 0 \\ 0 & y  \end{pmatrix}$ admits no geodesically compatible metric in the neighborhood of a point $(x,y)=(0,0)$.    The next theorem shows that {\it regular} collisions of eigenvalues are always allowed.

Following \cite{nij3}, we say that an operator $L$ is {\it $\gl$-regular}, if its adjoint orbit $\mathcal O_L = \{ PLP^{-1}~|~ \mbox{$P$ is invertible}\}$ has maximal dimension.  In simpler terms, this means that each eigenvalue $\lambda$ of $L$ admits only one linearly independent eigenvector  (equivalently, only one  $\lambda$-block in the Jordan normal form).
If $L$ is an operator on a manifold $\mathsf M$, then its eigenvalues can still collide without violating the $\gl$-regularity condition.  In Nijenhuis geometry,  scenarios of such collisions can be very different (see  \cite{nij3}).  However, regardless of any particular scenario,  we have the following general local result.

\begin{Theorem}\label{thm:1}
Let $L$ be a $\gl$-regular real analytic Nijenhuis operator.   Then (locally) there exists a pseudo-Riemannian metric $g$ geodesically compatible with $L$.  Moreover,  such a metric $g$ can be defined  explicitly in terms of the second companion form of $L$.
\end{Theorem}

The construction of such a metric $g$ is explained in Section \ref{sec:proofthm1}, see Proposition \ref{prop:3.1} and formula \eqref{id1}.  As already noticed, the statement of  Theorem \ref{thm:1} was known only near algebraically generic points of $L$ (e.g., \cite{tams}),  so the (local) existence of a geodesically compatible partner of $L$  is a new result for singular points. Note that understanding the behaviour of geodisically equivalent metrics near singular points  is fundamentally important for their global analysis on compact manifolds and  also played a decisive role in proving the projective Lichnerowicz-Obata conjecture \cite{BMR,Lichnerowicz}. 

\begin{Remark}{\rm Let us emphasise that Theorem  \ref{thm:1}  is essentially local in the sense that a $\gl$-regular Niejnehuis operator defined on a closed manifold $\mathsf M$ may not admit any geodesically compatible metric $g$ on the whole of $\mathsf M$  (although locally such a metric can be found near each point).  One of such examples is a complex structure  $J$ on a closed orientable surface $\mathsf M^2_g$ of genus $g\ge 2$.  In dimension 2, $J$ is $\gl$-regular, however it is known that  $\mathsf M^2_g$ cannot carry non-proportional geodesically equivalent metrics. 
}\end{Remark}

Our third theorem generalises Sinjukov-Topalov hierarchy theorem \cite{splitglue, sinjukov, topalov, topalov1} and  gives a complete description of geodesically compatible partners for $\gl$-regular Nijenhuis operators in terms of their symmetries   (both at algebraically generic and singular points).

Recall that an operator $M$ is called a {\it symmetry} of $L$ if these operators commute in algebraic sense, i.e. $LM = ML$, and  the following relation holds for any vector field $\xi$:
\begin{equation} \label{eq:M1}
 M[L\xi, \xi] + L[\xi, M\xi] - [L\xi, M\xi]=0.
 \end{equation}
 If $LM=ML$ and 
 \begin{equation} \label{eq:M2}
 \langle L, M\rangle (\xi,\eta) \overset{\mathrm{def}}{=} M[L\xi, \eta] + L[\xi, M\eta] - [L\xi, M\eta]  - LM[\xi,\eta] = 0 \qquad\mbox{for all $\xi,\eta$}, 
 \end{equation}
 then $M$ is called a {\it strong symmetry} in \cite{nij4}. 

The definition of   $\langle L, M\rangle$ in  \eqref{eq:M2}  is  essentially due to Nijenhuis  \cite[formula 3.9]{nijenhuis}; it  defines  a $(1,2)$-tensor field provided that $L$ and $M$ commute.  The l.h.s. of  \eqref{eq:M1} is just  $\langle L, M\rangle (\xi,\xi)$ so that \eqref{eq:M2} implies \eqref{eq:M1}.   Also notice that the Fr\"olicher-Nijenhuis bracket of $L$ and $M$ can be written as $\langle L,M\rangle + \langle M,L\rangle$ and the Nijenhuis torsion of $L$ coincides with $\langle L, L\rangle$  (up to sign). 
   
\begin{Theorem}\label{thm:2}
Let $L$ and $g$ be geodesically compatible. Assume that
$M$ is $g$-self-adjoint and is a strong symmetry of $L$, then $L$ and $gM:= (g_{is}M^s_j)$ are geodesically compatible.

Moreover, if $L$ is $\gl$-regular, then every metric $\tilde g$ geodesically compatible with $L$ is of the form $\tilde g = gM$, where $M$ is a (strong)\footnote{As proved in \cite{nij4}, every symmetry of a $\gl$-regular Nijenhuis operator is strong.} symmetry of $L$.
\end{Theorem}

\begin{Remark}  {\rm
The first part of this theorem in a  slightly different setting was proved by N. Sinjukov \cite{sinjukov}  for  $M=L$ and later reproved and applied by 
 P.  Topalov \cite{topalov, topalov1}, who also generalised the result for $M= L^{-1}$. The case
    $M= f(L)$ for a polynomial $f$ immediately follows from the case $M=L$ and both Sinjukov and Topalov considered the case $M= L^k$ for $k\in \mathbb{N}$. 
			For an arbitrary real analytic function, the results were generalised in \cite{splitglue}, see also \cite[Theorem 3]{topalovnew}.   For diagonalisable $\gl$-regular Nijenhuis operators $L$, every symmetry of $L$ has the form  $f(L)$ for some smooth function $f$  so all  metrics geodesically compatible with such an operator $L$ form generalised Sinjukov-Topalov hierarchy in the terminology  of \cite{splitglue}. 
				However,  if $\gl$-regular  $L$  contains non-trivial Jordan blocks, then there exist strong symmetries that cannot be presented in the form $f(L)$, see \cite{nij4}.  
}\end{Remark}

For a given Nijenhuis operator $L$, we define the operator fields $A_i$ by the following recursion relations
\begin{equation}\label{rec}
A_0 = \operatorname{Id}, \quad A_{i + 1} = L A_i - \sigma_i \operatorname{Id}, \quad i = 0, \dots, n - 1,     
\end{equation}
where functions  $\sigma_i$ are coefficients of the characteristic polynomial of $L$ numerated as below: 
\begin{equation} \label{def:sigma}
\chi_L(\lambda)=\det(\lambda\,\Id - L)=\lambda^n - \sigma_1\lambda^{n-1} - \dots - \sigma_n.
\end{equation}
Equivalently, the operators $A_i$ can be defined from the matrix relation
$$
\det (\lambda\,\Id-L) \cdot (\lambda\,\Id - L)^{-1} =  \lambda^{n-1} A_0 +  \lambda^{n-2} A_1 +\dots + \lambda A_{n-2} + A_{n-1}.
$$

Consider the following system of quasilinear PDEs  defined by these operators 
\begin{equation}
\label{eq:syst8}
\begin{aligned}
u_{t_1} &= A_1\, u_x,\\
  &\dots\\
u_{t_{n-1}} &= A_{n-1}\, u_x,
\end{aligned}
\end{equation}
with $u^i=u^i(x, t_1,...,t_{n-1})$ being unknown functions in $n$ variables and $u = (u^1,\dots, u^n)^\top$.

The system \eqref{eq:syst8} can be obtained within the framework of the general construction
introduced by F. Magri and P. Lorenzoni in \cite{ml, m}. In particular, it is consistent (in the real-analytic category) in the sense 
 that  for any initial curve
$\gamma(x)$ there exists a solution $u=u(x, t_1,\dots,t_{n-1})$ such that $u(x, 0, \dots, 0) = \gamma(x)$. 
In the case of a diagonal Nijenhuis operator $L$, the corresponding system satisfies the semihamiltonicity
condition of S.\,Tsarev \cite{Tsarev} and is weakly-nonlinear in the sense of B.\,Rozhdestvenskii et al \cite{rs}.
Such systems were studied and integrated in quadratures in the diagonal case by E.\,Ferapontov\footnote{Ferapontov's result is more general, he explicitly integrated all diagonal weakly-nonlinear semihamiltonian systems} \cite{fer, fer1,FF1997} and by K.\,Marciniak and M.\,Blaszak in \cite{MB2010}, see also \cite{Bl1, Bl2}. The general, not necessarily diagonalisable, case was done  in \cite{nij4}.

The next portion of our results concerns finite-dimensional reductions of system \eqref{eq:syst8}.  Various types of 
finite-dimensional reductions of infinite-dimensional nonlinear integrable systems have been investigated since
the middle of 70s, see e.g. \cite{AFW, BN, hone,  MB2010, Veselov}. Informally,  a finite-dimensional reduction of an integrable PDE system 
is a subsystem of it, which is finite-dimensional and still integrable. The first of the following two theorems states that the set of geodesics of any metric geodesically compatible with a $\gl$-regular Nijenhuis operator $L$ is invariant with respect to the flow of \eqref{eq:syst8}. That is, by fixing a metric $g$ geodesically compatible with $L$, we obtain a reduction of our infinite dimensional system to the set of  $g$-geodesics, which can be naturally endowed with the structure of a smooth  manifold of dimension $2n$.

  Next, we show that  the restriction of our system to the set of $g$-geodesics is equivalent, in the natural sense, to the Poisson action generated by the quadratic integrals of the geodesic flow which are closely related to the operators $A_i$ \cite{eqm, MT}.  Namely,  if $g$ is geodesically compatible with $L$, then its geodesic flow (as a Hamiltonian system on $T^*\mathsf{M}$)   admits $n$ commuting first integrals  $F_0,\dots, F_{n-1}$  of the form
\begin{equation}
\label{eq:integrals}
F_i(u,p) = \tfrac{1}{2}\, g^{-1} (A_i^* p, p). 
\end{equation}

\begin{Theorem}\label{thm:3a}  Consider any  metric $g$ geodesically compatible with $L$ and take any  geodesic  
 $\gamma(x)$ of this metric. Let  $u(x, t_1,...,t_{n-1})$ be the solution of \eqref{eq:syst8} with the initial condition $u(x,0,..., 0)=\gamma(x)$. Then  for any (sufficiently small) 
 $t_1,...,t_{n-1}$, the curve $x \mapsto   u(x,t_1,..., t_{n-1})$ is a geodesic of $g$. 
\end{Theorem}

In other words, the evolutionary system corresponding to any of the equations from  \eqref{eq:syst8} sends geodesics of $g$ to geodesics. 
Let us consider the space $\mathfrak{G}$ of all $g$-geodesics  (viewed as parameterised curves). This set has a natural structure of a $2n$-dimensional manifold. By Theorem \ref{thm:3a},  system  \eqref{eq:syst8} defines a local action 
of  $\mathbb{R}^n$ on $\mathfrak{G}$:
$$
\Psi^{t_0,t_1,\dots,t_{n-1}}:  \mathfrak{G} \to \mathfrak{G}, \qquad (t_0,t_1,\dots,t_{n-1})\in\R^n.
$$    
More precisely,  if $\gamma=\gamma(x) \in  \mathfrak{G}$ is a $g$-geodesic, then
we set $\Psi^{t_0,t_1,\dots,t_{n-1}}(\gamma)$ to be the $g$-geodesic $\tilde\gamma(x) =  u(x+t_0,t_1,..., t_{n-1})$,  where  $u(x, t_1,...,t_{n-1})$ is the solution of \eqref{eq:syst8} with the initial condition $u(x,0,..., 0)=\gamma(x)$.

\begin{Theorem}\label{thm:3b}   The action $\Psi$ is conjugate to the Hamiltonian action of $\mathbb{R}^n$ on $T^*{\mathsf{M}}$ generated by the flows of the integrals $F_0,\dots,F_{n-1}$ defined by \eqref{eq:integrals}. The conjugacy is given by $\gamma \in \mathfrak{G} \mapsto (\gamma(0), g_{ij}\dot\gamma^i(0)) \in T^*{\mathsf{M}}$. 
\end{Theorem}

\begin{Remark}{\rm 
Strictly speaking,  the action $\Psi$ is {\it local}, since the solutions $u(x, t_1,\dots, t_{n-1})$ of \eqref{eq:syst8} are, in general, defined only for small values of $t_i$ and $x$ from a certain (perhaps, small) interval.  However,  if the  Hamiltonian flows of $F_0,\dots, F_{n-1}$ are complete, then Theorem \ref{thm:3b} becomes global in the sense that the initial geodesic $\gamma(x)$ is defined for all $x\in\R$ and 
$(t_0,t_1,\dots, t_{n-1})\in\R^n$ is arbitrary.
}\end{Remark}

Let $L$ be a $\gl$-regular real analytic Nijenhuis operator, then it follows from \cite{nij4}  that  for every curve $\gamma$ with a cyclic velocity vector there exists a metric $g$ geodesically compatible with $L$ such that  $\gamma$ is a $g$-geodesic.
Thus, the finite-dimensional reductions of \eqref{eq:syst8}  provided by Theorems \ref{thm:3a} and \ref{thm:3b} `cover' almost all (local) solutions of the Cauchy problem. Indeed,  every generic initial curve belongs to a suitable set $\mathfrak G$ of geodesics from Theorem \ref{thm:3b}.


\section{Proof of Theorem \ref{thm:0}}

Recall that $L$ is assumed to be $g$-self-adjoint.
We use the geodesic compatibility condition for $g$ and $L$ in the form  \eqref{eq:geodeqv0}.
The operators in the left  hand side and the right hand side of  \eqref{eq:geodeqv0} are both $g$-self-adjoint, so that we can equivalently rewrite this relation as
\begin{equation}
\label{eq:geodeqv3}
g\bigl( (\nabla_\eta L) \xi, \xi\bigr) =    g( \eta  , \xi)  \, \mathcal L_\xi \tr L      \quad \mbox{for all tangent vectors $\eta$ and $\xi$.}   
\end{equation}
Then we have
$$
g\bigl( (\nabla_\eta L) \xi, \xi\bigr) = g \bigl(   \nabla_\eta (L\xi) - L\nabla_\eta \xi, \xi     \bigr)=
g \bigl( [\eta, L\xi] + \nabla_{L\xi} \eta   - L ([\eta,\xi] +\nabla_\xi \eta), \xi     \bigr)=
$$
$$
 g\bigl( \nabla_{L\xi} \eta, \xi\bigr)  -  g \bigl(\nabla_\xi \eta, L\xi\bigr)  + g \bigl( [\eta, L\xi], \xi \bigl) -  g \bigl( [\eta,\xi], L\xi\bigr)  =
$$
$$
\nabla_{L\xi} \bigl(g(  \eta, \xi)\bigr) - g\bigl(  \eta, \nabla_{L\xi}\xi\bigr)  -  \nabla_\xi \bigl(g ( \eta, L\xi)\bigr) + g \bigl( \eta, \nabla_\xi L\xi\bigr)  + g \bigl( [\eta, L\xi], \xi \bigl) -  g \bigl( [\eta,\xi], L\xi\bigr)  =
$$
$$
\nabla_{L\xi} \bigl(g(  \eta, \xi)\bigr) -  \nabla_\xi \bigl(g ( \eta, L\xi)\bigr) - g\bigl(  \eta, \nabla_{L\xi}\xi - \nabla_\xi L\xi\bigr)   + g \bigl( [\eta, L\xi], \xi \bigl) -  g \bigl( [\eta,\xi], L\xi\bigr)  =
$$
$$
\mathcal L_{L\xi} \bigl(g(  \eta, \xi)\bigr) -  \mathcal L_\xi \bigl(g ( \eta, L\xi)\bigr) - g\bigl(  \eta, [L\xi ,\xi]\bigr)   + g \bigl( [\eta, L\xi], \xi \bigl) -  g \bigl( [\eta,\xi], L\xi\bigr),
$$
which leads to \eqref{eq:geodeqv4bis}, as required.

In local coordinates, condition  \eqref{eq:geodeqv0} takes the form:
$$
\frac{\partial L^m_j}{\partial x^i} + \Gamma_{i\alpha}^m L^\alpha_j - \Gamma_{ij}^\alpha L_\alpha^m = \frac{1}{2} ( l_j \delta_i^m + g^{\alpha m} l_\alpha g_{ij}), \quad \mbox{where }  l_j = \frac{\partial \tr L}{\partial x^j}.
$$
Multiplying by $g_{km}$ (and summing over $m$) gives
$$
g_{km}\frac{\partial L^m_j}{\partial x^i} + \Gamma_{i\alpha, k} L^\alpha_j - g_{km}\Gamma_{ij}^\alpha L_\alpha^m = \frac{1}{2} ( l_j g_{ik} + l_k g_{ij}).
$$
Using the fact that $L$ is $g$-self-adjoint  (i.e.,  $g_{km}L^m_\alpha = g_{\alpha m}L^m_k$), we get (after cosmetic changes of some summation indices) 
$$
g_{k\alpha} \frac{\partial L^\alpha_j}{\partial x^i} + \Gamma_{i\alpha,k} L^\alpha_j - \Gamma_{ij,\alpha}L^\alpha_k = \frac{1}{2} ( l_j g_{ik} + l_k g_{ij}).
$$
In more detail,
$$
g_{k\alpha} \frac{\partial L^\alpha_j}{\partial x^i} + \frac{1}{2}\left( \frac{\partial g_{ik}}{\partial x^\alpha} +  \frac{\partial g_{k\alpha}}{\partial x^i}  -  \frac{\partial g_{i\alpha}}{\partial x^k} \right) L^\alpha_j - \frac{1}{2} \left(  \frac{\partial g_{i\alpha }}{\partial x^j}   + 
         \frac{\partial g_{j\alpha }}{\partial x^i}  - 
         \frac{\partial g_{ij}}{\partial x^\alpha}  \right)L^\alpha_k = \frac{1}{2} ( l_j g_{ik} + l_k g_{ij}).
$$
Since $L$ is $g$-self-adjoint, we can rewrite it as follows: 
\begin{equation}
\label{eq:qeodeqv2}
\frac{1}{2} g_{k\alpha} \frac{\partial L^\alpha_j}{\partial x^i} +
\frac{1}{2} g_{\alpha j} \frac{\partial L^\alpha_k}{\partial x^i} 
+ \frac{1}{2}\left( \frac{\partial g_{ik}}{\partial x^\alpha}   -  \frac{\partial g_{i\alpha}}{\partial x^k} \right) L^\alpha_j + \frac{1}{2} \left( \frac{\partial g_{ij}}{\partial x^\alpha}    
-   \frac{\partial g_{i\alpha }}{\partial x^j}  \right)L^\alpha_k = \frac{1}{2} ( l_j g_{ik} + l_k g_{ij}),
\end{equation}
which makes it clear that the left hand side and right hand side of this relation are both symmetric in indices $j$ and $k$. Up to the factor $\frac{1}{2}$, this relation coincides with \eqref{eq:c1bis},  as needed.

\section{Proof of Theorem \ref{thm:1}} \label{sec:proofthm1}

Let $L$ be a $\gl$-regular Nijenhuis operator.  We start with a purely algebraic construction leading to a metric $h$ that is geodesically compatible with $L$. This construction is performed in specific coordinates (and $h$ will depend on the choice of such coordinates).

First, fix second companion coordinates $u^1, \dots, u^n$ of $L$ so that 
$$
L = L_{\mathsf{comp2}} = \left(\begin{array}{ccccc}
     0 & 1 & 0 & \dots & 0  \\
     0 &  & 1 & \dots & 0  \\
     & & & \ddots & \\
     0 & 0 & 0 & \dots & 1 \\
     \sigma_n & \sigma_{n - 1} & \sigma_{n - 2} & \dots & \sigma_1
\end{array}\right).
$$
 Recall that for real analytic $\gl$-regular Nijenhuis operators such coordinates always exist \cite{nij3}\footnote{It is an open question whether such coordinates exist for any {\it smooth} $\gl$-regular Nijenhuis operator $L$. By \cite{nij3}, it is the case if $L$ is algebraically generic, i.e. its eigenvalues have constant multiplicities.}.

Let $p_1, \dots, p_n, u^1, \dots, u^n$ be the corresponding canonical coordinates on the cotangent bundle
and consider the following algebraic identity 
\begin{equation}\label{id1}
h_1 L^{n - 1} + \dots + h_n \operatorname{Id} = \Big(p_n L^{n - 1} + \dots + p_1 \operatorname{Id}\Big)^2.    
\end{equation}
The operator in the right hand side commutes with $L$ and therefore can be uniquely written as a linear combination of $\Id, L, \dots , L^{n-1}$ (this is a characteristic properties of $\gl$-regular operators \cite{nij4})  so that the functions $h_1, \dots, h_n$ are uniquely defined. In fact, they are quadratic in $p_1,\dots,p_n$ and their coefficients are polynomials in $\sigma_i$'s.  Using the fact that $L=L_{\mathsf{comp2}}$, one can easily check that  $h_n, \dots, h_1$  are the elements of the first row of the matrix $\Big(p_n L^{n - 1} + \dots + p_1 \operatorname{Id}\Big)^2$.  

The statement of Theorem \ref{thm:1}  follows from 

\begin{Proposition} \label{prop:3.1}
The quadratic function $h_1(u,p)$ defines a non-degenerate (contravariant) metric which is 
geodesically compatible with $L=L_{\mathsf{comp2}}$. 
\end{Proposition}

To prove this proposition we need to verify three conditions:
\begin{itemize}
\item[(i)]  the quadratic form $h_1(u,p) = \sum h_1^{\alpha\beta}(u) p_\alpha p_\beta$ is non-degenerate,
\item[(ii)]  $L$ is $h_1$-self-adjoint,
\item[(iii)]  $L$ and $h_1$ satisfy the geodesic compatibility condition.
\end{itemize}

\begin{Lemma}\label{m2}
The contravariant quadratic form $h_i$ is non-degenerate, i.e., $\det\Bigl( h^{\alpha\beta}_1 (u) \Bigr)\ne 0$.
\end{Lemma}
\begin{proof}
We write the r.h.s. of \eqref{id1} as
$$
\Big(p_n L^{n - 1} + \dots + p_1 \operatorname{Id}\Big)^2 = \sum_{i = 0}^{2n - 1} \Bigg(\sum_{k = 1}^{i + 1} p_k p_{i - k + 2}\Bigg) L^i.
$$
Hence, from \eqref{id1} we get
$$
\frac{\partial^2 h_1}{\partial p_i \partial p_j} L^{n - 1} + \dots + \frac{\partial^2 h_n}{\partial p_i \partial p_j} \operatorname{Id} = L^{i + j - 2}.
$$
This immediately implies that non-zero terms of $\Bigl(h^{\alpha \beta}_1\Bigr)$ are the only ones with $\alpha + \beta \geq n + 1$. Moreover, one can see that $h^{\alpha \beta}_1 = 1$ for $\alpha + \beta = n + 1$.  In other words, the matrix of $h_1$ has the form
$$
\Bigl( h^{\alpha \beta}_1\Bigr) = \left( \begin{array}{ccccc}
     0 & \dots & 0 & 0 & 1  \\
     0 & \dots & 0 & 1 & * \\
     0 & \dots & 1 & * & *  \\
     & \iddots & & & \\
     1 & \dots & * & * & * \\
\end{array}\right),
$$
and is obviously non-degenerate, as stated. \end{proof}

The second condition (ii) is algebraic and can be checked directly by using standard matrix algebra.  
We, however, will derive it differently.
For our next construction we will need an operator field $\widehat L$, which can be understood as a prolongation of $L$ to the cotangent bundle. In the given second companion coordinates $(p,u)$, we set
$$
\widehat L = \left(\begin{array}{cc}
     L^{\top} & 0  \\
     0 & L 
\end{array}\right). 
$$
The next Lemma provides some differential identities, which are crucial for our construction.

\begin{Lemma}\label{m3}
The functions $h_i$  from \eqref{id1} satisify the following identities
    \begin{equation}\label{id2}
\begin{aligned}
    & \widehat{L}^* \ddd h_i = \sigma_i \ddd h_1 + \ddd h_{i + 1}, \quad i = 1, \dots, n - 1, \\
    & \widehat{L}^* \ddd h_n = \sigma_n \ddd h_1 
\end{aligned}
\end{equation}
\end{Lemma}
\begin{proof}
If we consider $p_i$ to be (scalar) parameters, then the l.h.s. of \eqref{id1} is a symmetry of $L$. This implies (see \cite[Lemma 2.3]{nij4}) that
\begin{equation}\label{smash}
\begin{aligned}
& L^* \bar \ddd h_i = \sigma_i \bar \ddd h_1 + \bar \ddd h_{i + 1}, \quad i = 1, \dots, n - 1, \\
& L^* \bar \ddd h_n = \sigma_n \bar \ddd h_1.
\end{aligned}    
\end{equation}
Here $\bar \ddd h_i = \pd{h_i}{u^1} \ddd u^1 + \dots + \pd{h_i}{u^n} \ddd u^n$, that is, {\it one half} of the differential $\ddd h_i$ on the cotangent bundle.
Differentiating \eqref{id1} in $p_i$ we get
$$
\pd{h_1}{p_i} L^{n - 1} + \dots + \pd{h_n}{p_i} \operatorname{Id} = 2 L^{i - 1} \Big(p_n L^{n - 1} + \dots + p_1 \operatorname{Id}\Big), \quad i = 1, \dots, n.
$$
This implies (using the Cayley--Hamilton theorem)
$$
\begin{aligned}
    & \pd{h_1}{p_{i + 1}} L^{n - 1} + \dots + \pd{h_n}{p_{i + 1}} \operatorname{Id} =  L \Big( \pd{h_1}{p_{i}} L^{n - 1} + \dots + \pd{h_n}{p_{i}} \operatorname{Id} \Big) = \\
    & = \pd{h_1}{p_{i}} L^{n} + \dots + \pd{h_n}{p_{i}} L = \Big(\pd{h_2}{p_i} + \sigma_1 \pd{h_1}{p_i}\Big) L^{n - 1} + \dots + \Big(\pd{h_{n}}{p_i} + \sigma_{n - 1} \pd{h_1}{p_i}\Big) L + \sigma_n \pd{h_1}{p_i} \operatorname{Id}.
\end{aligned}
$$
As $L$ is $\gl$-regular, the coefficients in front of the powers of $L$ coincide, and we arrive to the system of $n - 1$ matrix equations
\begin{equation}\label{sys}
\pd{h}{p_{i + 1}} = L_{\mathsf{comp1}} \pd{h}{p_i}, \quad i = 1, \dots, n - 1.    
\end{equation}
Here $\pd{h}{p_i}$ is the column-vector  $\left(\pd{h_1}{p_i}, \dots, \pd{h_n}{p_i}\right)^{\top}$ and $L_{\mathsf{comp1}}$ denotes the first companion form of $L$  (obtained from $L_{\mathsf{comp2}}$ by transposition w.r.t. the anti-diagonal). Notice that \eqref{sys} can be rewritten in equivalent form $\pd{h}{p_{i + 1}} = (L_{\mathsf{comp1}})^i \pd{h}{p_1}$. 
In particular, $\pd{h}{p_{n}} = (L_{\mathrm{comp1}})^{n-1} \pd{h}{p_1}$ implying 
$$
L_{\mathsf{comp1}} \pd{h}{p_n} = L^n  \pd{h}{p_1} =    (\sigma_1 L^{n-1} + \dots + \sigma_n \Id) \pd{h}{p_1} =  \sigma_1\pd{h}{p_n} + \dots + \sigma_n \pd{h}{p_1}.
$$ 
Together with this additional equation, we can write \eqref{sys} in matrix form: 
$$
L_{\mathsf{comp1}} \, \left(\pd{h}{p}\right) = \left(\pd{h}{p}\right) \, L_{\mathsf{comp2}}^{\top},
$$
where $\left(\pd{h}{p}\right)$ is the Jacobi matrix of $h_1, \dots, h_n$ w.r.t.  $p_1,\dots, p_n$. 
Introducing $\tilde \ddd h_i =\left( \pd{h_i}{p_1} ,\dots, \pd{h_i}{p_n} \right)$ and recalling that in our coordinates $L = L_{\mathsf{comp2}}$, we can finally rewrite this system as 
\begin{equation}\label{mash}
\begin{aligned}
    & \tilde \ddd h_i \, L^{\top} = \sigma_i \tilde \ddd h_1 + \tilde \ddd h_{i + 1}, \quad i = 1, \dots, n - 1, \\
    & \tilde \ddd h_n \, L^{\top} = \sigma_n \tilde \ddd h_1.
\end{aligned}
\end{equation}
Gathering \eqref{mash} and \eqref{smash} we get \eqref{id2} from the statement of the Lemma.
\end{proof}

\begin{Corollary}
$L$ is $h_1$-self-adjoint.
\end{Corollary}

\begin{proof}  By saying that $L$ is self-adjoint w.r.t. the contravariant metric $h_1$, we mean that 
$h_1 (L^*\alpha, \beta) = h_1(\alpha, L^*\beta)$ for any 1-forms (co-vectors) $\alpha$ and $\beta$. In local coordinates, this condition means that the tensor $h_1^{ks}L_s^j$ is symmetric in indices $k$ and $j$.  This is, of course, equivalent to the fact $L$ is self-adjoint w.r.t. the covariant metric $h_1^{-1}$. 

For $i=1$, relations \eqref{mash} give $\tilde \ddd h_1 \, L^{\top} = \sigma_1 \tilde \ddd h_1 + \tilde \ddd h_{2}$.  In coordinates, 
\begin{equation}
\label{eq:FviaH}
h_1^{ks}L_s^j = \sigma_1 h_1^{kj} + h_2^{kj}
\end{equation}  
and we see that the l.h.s. is indeed symmetric  in $k$ and $j$, since $h_1$ and $h_2$ in the r.h.s. are both symmetric by construction.  This completes the proof. \end{proof}

The standard geodesic compatibility condition \eqref{eq:geodeqv0} is not very convenient in our current setting as we deal with a {\it contravariant} metric $h_1$.   However,  there is another elegant Benenti condition which is quite suitable for our purposes,  see \cite[Definition 1 and Theorem 1]{eqm}.  Namely, $h_1$ and $L$ will be geodesically compatible if the quadratic functions 
\begin{equation}
\label{eq:functions}
H  = \frac{1}{2} h_1(p,u) =   \frac{1}{2}  h_1^{ij}(u) p_i p_j \quad\mbox{and}\quad
F = h_1^{ks} L_s^j p_k p_j
\end{equation}
satisfy the following commutation relation  w.r.t. the standard Poisson structure on $T^*{\mathsf{M}}$
\begin{equation}
\label{eq:geodeqvcontra}
\{ H, F\} =  2H \cdot \left( \tfrac{\partial \tr L}{\partial u^q} h_1^{\alpha q} p_\alpha   \right).
\end{equation}
  
To verify this relation we need the following algebraic lemma.

\begin{Lemma}\label{m1}
Let $\mathcal P$ be a skew-symmetric form on a vector space $V^n$ and $L$ be an $\mathcal P$-symmetric operator, that is, $\mathcal P(L\xi, \eta) = \mathcal P(\xi, L\eta)$ for all vectors $\xi, \eta$. Then for all integer $p, q \geq 0$ and arbitrary $\xi$ one has
$$
\mathcal P(L^p \xi, L^q \xi) = 0.
$$
\end{Lemma}
\begin{proof}
The above symmetry condition implies that the form $\mathcal P_L (\xi, \eta) = \mathcal P(L\xi, \eta)$ is skew-symmetric.
Now, if $p + q = 2k$, then due to $L$ being $\mathcal P$-symmetric, we get $\mathcal P(L^p\xi, L^q\xi) = \mathcal P(L^k\xi, L^k \xi) = 0$. Similarly for $p + q = 2k + 1$,  we get $\mathcal P(L^p\xi, L^q\xi) = \mathcal P(L^{k + 1}\xi, L^k \xi) = \mathcal P_L (L^{k}\xi, L^k \xi) = 0$ as $\mathcal P_L$ is skew-symmetric. Lemma is proved.
\end{proof}

Now notice that by construction $\widehat L^*$ is $\mathcal P$-symmetric with respect to the Poisson structure $\mathcal P = \Omega^{-1}$. On the other hand, relations \eqref{id2} imply that $\ddd h_1, \dots, \ddd h_n$ belong to  the subspace spanned by $\bigl(\widehat L^*\bigr)^{k} \ddd h_1$ $(k = 0, 1, \dots)$. By Lemma \ref{m1},  this subspace is isotropic w.r.t. $\mathcal P$, which means that 
$$
\{h_i, h_j\} = 0,
$$
where $\{\, , \, \}$ in the standard Poisson bracket on the cotangent bundle.

To finish the proof it remain to notice that relation \eqref{eq:FviaH} means that the second function $F$ from \eqref{eq:functions}  can be written as $F = h_2 + \sigma_1 h_1$, where $\sigma_1= \tr L$.  Due to Poisson commutativity of $h_1$ and $h_2$ we have
$$
\{H, F\} = \{\tfrac{1}{2} h_1,  h_2 + \sigma_1 h_1\}= \tfrac{1}{2} h_1 \cdot  \{h_1,\sigma_1\}  = H  \cdot \{h_1, \tr L \} =2 H\cdot \Big( \pd{\operatorname{tr} L}{u^q} h^{\alpha q}_1 p_\alpha\Big),
$$
which coincides with \eqref{eq:geodeqvcontra}. This completes the verification of conditions (i), (ii), and (iii), and hence   the proof of Theorem \ref{thm:1}.

\section{Proof of Theorem \ref{thm:2}}

Let $L$ and $g$ be geodesically compatible.  We start with the first statement of Theorem \ref{thm:2}  and consider a $g$-self-adjoint operator  $M$ which is a strong symmetry of $L$. We need to show that the geodesic compatibility condition still holds if we replace $g$ with $\tilde g = g M$  (i.e. $\tilde g(\eta,\xi) =g(M\eta,\xi)$).

We use this condition in the form  \eqref{eq:geodeqv4bis}.  We have
$$
\begin{aligned}
&\mathcal L_{L\xi} \bigl(\tilde g (\eta,\xi)\bigr)  - \mathcal L_\xi \bigl(\tilde g (\eta,L\xi)\bigr) + \tilde g \bigl(\eta, [\xi, L\xi]\bigr) + \tilde g\bigl([\eta, L\xi],\xi\bigr) - \tilde g\bigl(L[\eta, \xi],\xi\bigr) = \\
& \mathcal L_{L\xi} \bigl(g (M\eta,\xi)\bigr)  - \mathcal L_\xi \bigl(g (M\eta,L\xi)\bigr) + g \bigl(M\eta, [\xi, L\xi]\bigr) + g\bigl(M[\eta, L\xi]- ML[\eta, \xi],\xi\bigr) = \\
&   \mathcal L_{L\xi} \bigl(g (M\eta,\xi)\bigr)  - \mathcal L_\xi \bigl(g (M\eta,L\xi)\bigr) + g \bigl(M\eta, [\xi, L\xi]\bigr) + {\color{blue} g\bigl([M\eta, L\xi]- L[M\eta, \xi],\xi\bigr)} + \\
&   g\bigl(M[\eta, L\xi]- ML[\eta, \xi] {\color{blue} - [M\eta, L\xi]  + L[M\eta,\xi]},\xi\bigr) =  \\
&   g \bigl((\nabla_{M\eta} L) \xi, \xi\bigr)  + g \bigl(\langle M , L\rangle (\eta,\xi) , \xi\bigr) =  g \bigl((\nabla_{M\eta} L) \xi, \xi\bigr) + 0  = \\
&  g(M\eta,\xi)\, \mathcal L_\xi \tr L =  \tilde g(\eta,\xi) \,  \mathcal L_\xi \tr L,
\end{aligned}
$$
as required. This proves the first statement of Theorem \ref{thm:2}.

\begin{Remark}\label{rem1}{\rm
This computation also leads to the following conclusion.  Let $g$ and $L$ be geodesically compatible. Then a metric  $\tilde g = gM$ is geodesically compatible with $L$ if and only if  
$g \bigl(\langle M , L\rangle (\eta,\xi) , \xi\bigr)$
vanishes for all $\eta, \xi$.  
}\end{Remark}

Now assume that $L$ is $\gl$-regular. To prove the second statement of Theorem \ref{thm:2} we need the following algebraic fact.

\begin{Lemma} \label{lem:4.1}
Let operators $L$ and $M$ be $g$-self-adjoint.   Then $L$ is self-sdjoint with respect to the metric $\tilde g = gM$ if and only if  $L$ and $M$ commute (i.e., $ML- LM=0$).
\end{Lemma}
\begin{proof}
The statement immediately follows from the identity
$$
\tilde g(L \xi, \eta) - \tilde g(\xi , L\eta)= g(ML \xi, \eta) - g(M\xi, L\eta) =  g\bigl((ML - LM) \xi, \eta\bigr).
$$
Since $g$ is non-degenerate,  the vanishing of the l.h.s. is equivalent to $ML - LM = 0$, as stated.
\end{proof}

Now let $L$ be a $\gl$-regular Nijenhuis operator which is geodesically compatible with a metric $g$   (notice that such a metric always exists  by Theorem \ref{thm:1}).

Let $\tilde g$ be another metric geodesically compatible with $L$. Recall that $L$ is $\tilde g$-self-adjoint by definition. Define $M$ to be the operator field that relates these two metrics, that is, $\tilde g = gM$ so that $M$ is automatically $g$-self-adjoint. 
By Lemma \ref{lem:4.1}, $ML - LM = 0$ and, thus, $M = f_1 L^{n - 1} + \dots + f_n \operatorname{Id}$  for some smooth functions $f_1,\dots,f_n$. We introduce the following tensor of type $(1, 2)$ which we treat as a vector-valued bilinear form:
$$
T_M = \ddd f_n \otimes L^{n - 1} + \dots + \ddd f_n \otimes \operatorname{Id} 
$$
By definition
\begin{equation}\label{rm0}
    L T_M(\xi, \eta) = T_M(\xi, L \eta).
\end{equation}
As $L$ is  $g$-self-adjoint, we have
\begin{equation}\label{rm1}
g\bigl(T_M(\xi, \eta), \zeta\bigr) = g\bigl(\eta, T_M(\xi, \zeta)\bigr)    
\end{equation}
for all vectors $\xi, \eta, \zeta$. By straightforward computation, using the fact that $\langle L^i, L \rangle = 0$, we get
\begin{equation}\label{rm2}
\begin{aligned}
    \langle M, L\rangle (\eta, \xi) = L [M\eta, \xi] + M [\eta, L\xi] & - LM[\eta, \xi] - [M\eta, L\xi] = \\
    & = T_M(\xi, L\eta) - T_M (L\xi, \eta).
\end{aligned}    
\end{equation}
From \eqref{rm0}, \eqref{rm1} and \eqref{rm2} we obtain
\begin{equation}
\label{eq:rm3}
\begin{aligned}
   & g\bigl(\langle M, L \rangle (\eta, \xi), \xi\bigr) = \\  =\, & g\bigl(T_M(\xi, L\eta), \xi\bigr) - g\bigl(T_M(L\xi, \eta), \xi\bigr) = g\bigl(L\eta, T_M(\xi, \xi)\bigr) - g\bigl(\eta, T_M(L\xi, \xi )\bigr) = \\
      = \, & g\bigl(\eta, LT_M(\xi, \xi)\bigr) - g\bigl(\eta, T_M(L\xi, \xi )\bigr) = g\bigl(\eta, T_M(\xi, L\xi) - T_M(L\xi, \xi )\bigr) = \\
      = \, & - g \bigl(\eta, \langle M, L \rangle (\xi, \xi)\bigr)
\end{aligned}
\end{equation}
By Remark \ref{rem1},  $\tilde g=gM$ is geodesically compatible with $L$ if and only if 
$$
g\bigl(\langle M, L \rangle (\eta, \xi), \xi\bigr) = 0,
$$
for all $\xi, \eta$.  In view of \eqref{eq:rm3},  this implies that $\langle M, L \rangle (\xi, \xi) = 0$, i.e., $M$ is a symmetry of $L$.  Since $L$ is $\gl$-regular, the symmetry $M$ is strong \cite[Theorem 1.2]{nij4}, as required.

\section{Proof of Theorems \ref{thm:3a} and \ref{thm:3b}}

Recall that a symmetric  $(0,2)$-tensor   $A=A_{ij}$ is called a Killing tensor for a (pseudo)-Riemannian metric $g$ if it satisfies the following condition:
$$
\nabla_k A_{ij} + \nabla_i A_{jk} + \nabla_j A_{ki} = 0,
$$ 
where $\nabla$ denote the Levi-Civita connection of $g$. We will  also refer to  the operator $A^i_j = g^{is}A_{sj}$   obtained from $A=A_{ij}$ by raising index as  Killing $(1,1)$-tensor.  Recall that the equivalent definition for Killing $(1,1)$-tensors is as follows:   $A=A^i_j$ is Killing $(1,1)$-tensor for a metric $g$ if $A$ is $g$-self-adjoint and the Hamiltonians 
\begin{equation}
\label{eq:Hamiltonians}
H(u,p)= \frac{1}{2} g^{ij}(u) p_i p_j\quad \mbox{and} \quad F(u,p)=\frac{1}{2} A^i_\alpha g^{\alpha j}(u)p_i p_j
\end{equation}
commute on $T^*\mathsf{M}$ w.r.t. the canonical Poisson structure.

Theorems \ref{thm:3a} and \ref{thm:3b}  are based on the following general statement that establishes a natural relation between quadratic Killing tensors of (pseudo)-Riemannian metrics and solutions of some quasilinear systems. Its special case is   \cite[Remark 2 and Proposition 3]{MB2010}, see also \cite{Bl1, Bl2}.

\begin{Proposition}
\label{prop:basic}
Let $A=A^i_j$ be a Killing $(1,1)$-tensor for a Riemannian metric $g$.  Consider the Hamiltonian $\R^2$-action on  $T^*\mathsf{M}$ generated by the Poisson commuting Hamiltonians \eqref{eq:Hamiltonians}:
$$
\Phi^{x,t}=\Phi^x_H\circ \Phi^t_F:  T^*{\mathsf{M}} \to T^*{\mathsf{M}}.
$$  
Then for any initial condition $(u_0, p_0)$,    the function $u(t,x)$ defined from the relation $\Phi^{x,t} (u_0,p_0)=
\bigl(u(x,t), p(x,t)\bigr)$
is a solution of the quasilinear system 
\begin{equation}
\label{sys1}
u_t = A(u) u_x.
\end{equation}
\end{Proposition} 

\begin{proof}
By definition, the functions $\bigl(u(x,t), p(x,t)\bigr)$ define the natural parametrisation of the $\Phi$-orbit of the point $(u_0, p_0)$.   In particular,   $\bigl(u(x,t_{\mathrm{c}}), p(x,t_{\mathrm{c}})\bigr)$ for a fixed $t=t_{\mathrm{c}}$ is a solution of the Hamiltonian system generated by $H$ and, similarly,  $\bigl(u(x_{\mathrm{c}},t), p(x_{\mathrm{c}},t)\bigr)$ for a fixed $x=x_{\mathrm{c}}$ is a solution of the Hamiltonian system generated by  $F$. Hence at each point $(x,t)$, we have
$$
\frac{\partial u^i}{\partial x} = \frac{\partial H}{\partial p_i}= g^{ij} p_j\quad\mbox{and similarly}\quad 
\frac{\partial u^i}{\partial t} = \frac{\partial F}{\partial p_i}= A^i_\alpha g^{\alpha j} p_j,
$$
which immediately implies $\frac{\partial u^i}{\partial t}  =  A^i_\alpha  \frac{\partial u^\alpha}{\partial x}$ or, shortly, 
$u_t = A(u) u_x$, as required.
\end{proof}

This proposition can be naturally generalised to the case when $g$ admits several {\it commuting} Killing tensors.
Indeed,  consider a (pseudo)-Riemannian metric $g$ and $g$-self-adjoint operators  $A_0=\Id, A_1, \dots , A_{n-1}$ such that the quadratic functions $F_i (x,p) =  \frac{1}{2} g^{-1} ( A_i^* p, p)$, $i=0,1,\dots,n-1$, pairwise commute on $T^*\mathsf{M}$ w.r.t. the canonical Poisson structure (in particular, each of $A_i$ is a Killing $(1,1)$-tensor for $g$).
These functions generate a Hamiltonian $\R^{n}$-action on $T^*{\mathsf M}$:
$$
\Phi^{(x,t_1,\dots,t_k)} :  T^*{\mathsf{M}} \to T^*{\mathsf{M}},    \quad \Phi^{(x,t_1,\dots,t_k)}  = \Phi^x_{F_0}\circ  \Phi^{t_1}_{F_1} \circ \dots \circ \Phi^{t_k}_{F_k},
$$ 
where $\Phi^{t}_{F_i}$ denotes the Hamiltonian flow generated by $F_i$. 
 
\begin{Corollary}\label{cor:1}
In the above setting, let 
$$
\bigl(u(x,t_1,\dots,t_{n-1}), p(x,t_1,\dots,t_{n-1})\bigr) = \Phi^{(x,t_1,\dots,t_{n-1})} (u_0, p_0)
$$ 
be an orbit of this action. Then $u(x,t_1,\dots,t_{n-1})$ is a solution of the system of quasilinear equations
\begin{equation}
\label{sys2}
\begin{aligned}
u_{t_1} &= A_1(u) u_x,\\
u_{t_2} &= A_2(u) u_x,\\
&\dots \\
u_{t_{n-1}} &= A_{n-1}(u) u_x.
\end{aligned}
\end{equation}
\end{Corollary}

The statement of Theorem \ref{thm:3a} follows from this Corollary and the fact that the operators $A_i$ from system  \eqref{eq:syst8} are commuting Killing $(1,1)$-tensors for any metric $g$ that is geodesically compatible with $L$  \cite[Corollary 1]{eqm}. Indeed, if $\gamma (x)= (u^1(x),\dots, u^n(x))$ is a geodesic of $g$, then the curve $(u (x), p(x))$ with $p_i(x)  = g_{ij}(u) \dot u^j(x)$ is an orbit of the Hamiltonian flow $\Phi^x_{F_0}$ on $T^*{\mathsf{M}}$ generated by $F_0=H = \frac{1}{2} g^{ij}(u) p_i p_j$.   This $\R^1$-orbit can be naturally included into an $\R^n$-orbit of the action  $\Phi^{(x,t_1,\dots,t_{n-1})}$  so that    $\bigl(u(x,0,\dots, 0), p(x, 0,\dots, 0)\bigr)=(u(x), p(x))$.

Now for any fixed $t_1, \dots, t_{n-1}$, the curve  $\bigl(u(x,t_1,\dots,t_{n-1}), p(x,t_1,\dots,t_{n-1})\bigr)$ parametrised by $x$ is still an
orbit of the Hamiltonian flow $\Phi^x_{F_0}$.  Since $F_0 = H$ is the Hamiltonian of the geodesic flow of $g$, the curve $u(x,t_1,\dots,t_{n-1})$ is a $g$-geodesic  (for fixed $t_1, \dots, t_{n-1}$). On the other hand,  by Corollary \ref{cor:1},  $u(x,t_1,\dots,t_{n-1})$ is the solution of \ref{thm:3a} with the initial conditon $u(x,0,\dots,0)= \gamma(x)$.   This completes the proof of  Theorem \ref{thm:3a}.

Thus, the evolutionary flows $u_{t_i} = A_i(u) u_x$ naturally act on the space of $g$-geodesics.  Since every geodesic is uniquely defined by its initial condition $(u(0), p(0))\in T^*\mathsf{M}$,   we can naturally identify the space $\mathfrak G$ of all (parametrised) geodesics with the cotangent bundle $T^*\mathsf{M}$ by setting  
\begin{equation}
\label{eq:conjugacy}
\gamma \in \mathfrak G \mapsto (u(0), p(0)) \in T^*\mathsf{M}, 
\end{equation}
where
$u(0)= \gamma(0)$ and $p_i(0) = g_{ij}(u(0)) \dot u^j$. 

To prove Theorem \ref{thm:3b},  it remains to compare the action $\Psi^{t_0,\dots,t_{n-1}}$ on $\mathfrak G$ and the Hamiltonian action $\Phi^{t_0,\dots,t_{n-1}}$ on $T^*\mathsf{M}$. We have
$$
\Psi^{t_0,\dots,t_{n-1}} (\gamma(x)) =  u(x+t_0, t_1,\dots, t_{n-1}) \quad\mbox{and}\quad
\Phi^{t_0,\dots,t_{n-1}} (u(0), p(0)) = (u(t_0,\dots,t_{n-1}), p(t_0,\dots,t_{n-1})).
$$  

Taking into account that $p_i(t_0,\dots,t_{n-1}) = g_{ij} \frac{\ddd}{\ddd x} |_{x=0} u^j(t_0+x,t_1,\dots,t_{n-1})$, we see that the map 
\eqref{eq:conjugacy} indeed conjugates the actions $\Psi$ and $\Phi$. This completes the proof of Theorem \ref{thm:3b}.

\end{document}